\documentclass[a4paper,12pt,epsfig]{article}
\usepackage{amsmath,amssymb,amsthm,amscd}
\usepackage{graphics}

\usepackage{url}
\newcommand{\cp}{\rm Cap}
\newcommand{\bv}{\rm BV}
\newcommand{\dg}{\rm DG}

\newtheorem{thm}{Theorem}
\newtheorem{prop}{Proposition}
\newtheorem{cor}{Corollary}

\newtheorem{defn}{Definition}

\title {An extension the semidefinite programming bound for spherical codes}

\author {Oleg R. Musin}

\begin{document}
\date{}
\maketitle

\begin{abstract}    In this paper we  present an  extension of known semidefinite and linear programming upper bounds for spherical codes and consider a version of this bound for distance graphs. We apply the main result for the distance distribution of a spherical code.  
\end {abstract}



\medskip

\section{Introduction}

Let  $C$ be an $N$--element subset of the unit sphere ${\mathbb S}^{n-1}\subset {\mathbb R}^{n}$. We say that $C$  is an $(N,n,\theta)$ {\em spherical code}  if the angular distance between every two distinct points of $C$ is at least $\theta$, in other words,  if every distinct points $(c,c')$ of $C$ have inner product $c\cdot c'$ at most $t:=\cos{\theta}$.  


Let $f(x)=\sum{f_kG_k^{(n)}(x)}$ be a non-negative linear combination of the Gegenbauer polynomials $G_k^{(n)}$ with $f_0>0$  such that  $f(x)\le 0$  for all $x\in [-1,t]$.  Then  for every $(N,n,\theta)$ spherical code we have: 
$$N \le \frac {f(1)}{f_0}. \eqno (1)$$
 Denote by $A(n,\theta)$  the maximal size $N$ of an $(N,n,\theta)$ spherical code.  Then (1) is equivalent to the following bound: 
 $$A(n,\theta)\le \frac {f(1)}{f_0}. $$
This bound  is called the {\em linear programming (LP)} or {\em Delsarte's bound} for spherical codes. 

\medskip

 The spherical cap with center $e\in {\mathbb S}^{n-1}$ and angular radius $\phi$ is the set 
 $$
 \cp(e,\phi):=\{x\in {\mathbb S}^{n-1}: e\cdot x\ge\cos{\phi}\}. 
 $$        
The maximal size of an $(N,n,\theta)$ spherical code in $ \cp(e,\phi)$  is denoted by $A(n,\theta,\phi)$ \cite{BM}.  


\medskip
 
  Let $f$ be a real function on the interval $[-1,-\cos{\phi}]$.  Let  $m\le A(n,\theta,\phi)$. Denote by $Q(m,n,\theta,\phi)$ the set of all $(m,n,\theta)$ spherical codes in $ \cp(e,\phi)$.  
  
Let $Y =\{y_1,\ldots,y_m\}\in Q(m,n,\theta,\phi)$, 
 $$
H_f(Y):=f(-e\cdot y_1)+\ldots+f(-e\cdot y_m),
$$ 
$$
h_m=h_m(n,\theta,\phi,f):=\sup\limits_{Y\in Q(m,n,\theta,\phi)}\{H_f(Y)\}, 
$$
$$
\hat h(n,\theta,\phi,f):=\max\{h_1,\ldots,h_\mu\}), \quad \mu:=A(n,\theta,\phi). 
$$

\medskip

  In \cite{mus08a} we found an extension of Delsarte's bound (1). Let $f$ be a non-negative linear combination of the Gegenbauer polynomials $G_k^{(n)}$ with $f_0>0$  such that  $f(x)\le 0$  for all $x\in [-\cos{\phi},\cos{\theta}]$. Then  for every $(N,n,\theta)$ spherical code \cite[Theorem 1]{mus08a} states that 
 $$
 N \le \frac {f(1)+\hat h(n,\theta,\phi,f)}{f_0}. \eqno (2)
 $$
 
 In \cite{mus08a} we applied this bound to prove that the kissing number in four dimensions is 24. Namely, we found $f$ with $f_0=1$ and $\phi$ such that 
 $$
 f(1)+\hat h(4,\pi/3,\phi,f)<25. 
 $$
(The proof of this inequality is the most difficult part of \cite{mus08a}.) Then from (2) follows that $k(4):=A(4,\pi/3)<25$. Since $k(4)\ge24$, we have that $k(4)=24$. 
 
 \medskip
  
The semidefinite programming (SDP) method for spherical codes was proposed by Bachoc and Vallentin \cite{bac08a} with further applications and extensions in \cite{bac09a,bac09b, cohn12, mac16,mv10,mus14}.

The positive--semidefinite property of the Gegenbauer polynomials yields  the positive--semidefinite property of matrices $S_k^n$. Now  consider polynomials $F$ that were defined by Bachoc and Vallentin. Let $F (x, y, z)$ be a symmetric polynomial with expansion
$$
F (x, y, z) =\sum\limits_{k=0}^d {\langle M_k, S^n_k (x, y, z)\rangle}
$$
in terms of the matrices $S^n_k$. 

\begin{defn}
For a given $f_0>0$ denote by $\bv(n,f_0)$ the class of symmetric polynomials  $F(x,y.z)$ that satisfy the following properties: 
\begin{enumerate}
\item all matrices $M_k$ are positive semidefinite,
\item $M_0 - f_0E_0$ is positive semidefinite ($E_0$ is the matrix whose only nonzero entry is the top left corner which
contains 1). 
\end{enumerate}

\end{defn}

Let 
 $$
D(\theta):=\left\{(x,y,z): -1\le x, y, z\le  \cos{\theta}, \, 1+2xyz-x^2-y^2-z^2\ge0\right\}.
 $$
{  Let (1) $F\in \bv(n,f_0)$,  (2) $F (x, x, 1) \le B$ for all $x \in [-1,  \cos{\theta}]$, and (3) $F (x, y, z)\le 0$ for all $(x, y, z) \in D(\theta)$.
Then (see \cite{bac08a})  for every $(N,n,\theta)$ spherical code we have: }
$$
N^2\le \frac{F(1,1,1)+3(N-1)B}{f_0}. \eqno (3)
$$

 \medskip


 \section{New SDP bound for spherical codes}
 
 In this section we prove a theorem that is an SDP generalization of  \cite[Theorem 1]{mus08a} and an extension of \cite[Theorem 4.1]{bac09a}.

\begin{defn}
 Let  $T$ be a subset of the interval  $[-1,1)$. Let $e\in{\mathbb S}^{n-1}$. 
Denote by  $A(n,\theta,T)$ the maximal size  of an $(N,n,\theta)$ spherical codes $C$ such that  for every $c\in C$ the inner product $c\cdot e$ belongs to $T$.  
 \end{defn}

  \begin{defn}
 Let  $g$ be a real function on  $T\subset[-1,1)$. Define 
 $$ g_T(x):=\left\{
\begin{array}{l}
g(x) \; \mbox{ for all } x\in T \\
0 \;  \mbox{ otherwise}
\end{array} 
\right.
$$
 \end{defn}

 \begin{defn}
 Let $g$ be a real function on  $T\subset[-1,1)$. Let  $m\le A(n,\theta,T)$. Let $e\in{\mathbb S}^{n-1}$.  Denote by $Q(m,n,\theta,T)$ the set of all $(m,n,\theta)$ spherical codes $C$ such that for every $c\in C$ the inner product $c\cdot e$ belongs to $T$.  
  
Let $Y =\{y_1,\ldots,y_m\}\in Q(m,n,\theta,T)$, 
 $$
H_g(Y):=g(e\cdot y_1)+\ldots+g(e\cdot y_m),
$$ 
$$
h_m=h_m(n,\theta,T,g):=\sup\limits_{Y\in Q(m,n,T)}\{H_g(Y)\}, 
$$
$$
\hat h(n,\theta,T,g):=\max\{h_1,\ldots,h_\mu\}, \quad \mu:=A(n,\theta,T). 
$$
 \end{defn}

\begin{defn}
For given $n,\, f_0$, $T\subset[-1,1)$, $g:T\to {\mathbb R}$, $B$ and $\theta$ denote by ${\mathcal F}(n,f_0,T,g,B,\theta)$ the class of symmetric polynomials  $F(x,y.z)$ that satisfy the following properties: 
\begin{enumerate}
\item $F\in \bv(n,f_0)$,
\item $F (x, x, 1) \le B+2g_T(x)$ for all $x\in[-1,\cos{\theta}]$,  
\item $F (x, y, z)\le g_T(x)+g_T(y)+g_T(z)$ for all $(x, y, z) \in  D(\theta)$. 
\end{enumerate}

\end{defn}

\begin{thm}   \label{thm1}   Let $F\in {\mathcal F}(n,f_0,T,g,B,\theta)$. Then an $(N,n,\theta$) spherical code satisfies the following inequality
$$
N^2\le \frac{F(1,1,1)+3(N-1)\,B+3N\,\hat h(n,\theta,T,g)}{f_0}. \eqno (4)
$$
\end{thm}

\begin{proof} 
Let $C$ be an $(n, N, \theta$) spherical code. Define 
$$
S:=\sum\limits_{(c,c',c'')\in C^3} {F(c\cdot c',c\cdot c'',c'\cdot c'')}.
$$
It is easy to see, that the positive semidefinite assumption on $F$ yields 
$$
S\ge N^3f_0. 
$$ 

On the other hand, the contribution of all triples $(c,c,c)$, $c\in C$,   to $S$ is $NF(1, 1, 1)$. Consider all triples with two pairwise different elements. Since $F(c,c,c')\le B+2g_T(c\cdot c')$, we see that the contribution  of pairs $(c,c)$ to $S$  is at most $3N(N-1)B$.

Let $(c,c')$, $c'\ne c$,  be an ordered pair. The number of all triples in $C^3$ that contains this pair is $3N$.  Consider in $C$ all points  $c_1,\ldots,c_m$ such that $c\cdot c_i\in T$.  Then 
$$
\sum\limits_{c'\in C,\, c'\ne c}{g_T(c\cdot c')}\le \sum\limits_{i=1}^m {g(c\cdot c_i)}\le h_m\le \hat h. 
$$
Hence the contribution  of all pairs $(c,c')$, $c'\ne c$, to $S$ is at most $3N^2\hat h$. 
Together,
$$
N^3f_0\le S \le NF(1, 1, 1)+3N(N-1)B+ 3N^2\,\hat h.
$$
\end{proof}

\medskip

Note that (4) extends (2) and (3).

(i)  Let $T=[-1,-\cos{\phi}]$. Now we have $\hat h(n,\theta,T,g)=\hat h(n,\theta,\phi,g)$. If 
$$F(x, y, z)=g(x)+g(y)+g(z), \quad B=g(1), \quad f(x)=3g(x),$$ 
then (4) becomes (2). 

(ii)  Let $g(x)=0$ for all $x$.  Then (4) becomes (3), i.e.  Theorem \ref{thm1} extends \cite[Theorem 4.1]{bac09a}.




\section{Bounds for the distance distribution}

Let $C$ be an $(N,n,\theta$) spherical code. The {\em distance distribution} of $C$ with respect to $u\in C$ is the system of numbers $\{A_t(u): -1\le t \le 1\}$, where 
$$
A_t(u):=|\{v\in C: v\cdot u=t\}|,
$$ 
and the {\em distance distribution} of $C$ is the system of numbers $\{A_t: -1\le t \le 1\}$, where 
$$
A_t:=\frac{1}{N}\sum\limits_{u\in C}{A_t(u)}. 
$$ 

Let $s:=\cos{\theta}$. It is clear the $A_t$ satisfy $A_1=1$, $A_t=0$ for $s<t<1$, and 
$$\sum\limits_{-1\le t\le s}{A_t}=N-1. $$

 Let  $T \subset[-1,1]$. Denote 
 $$
 A(T):=\sum\limits_{t\in T: A_t>0}{A_t}. 
 $$
Then $A(\{1\})=A_1=1$, $A([-1,1])=N$ and $A([-1,s])=N-1$. 

\medskip

Now we apply Theorem \ref{thm1} for the distance distribution of a spherical code $C$. 

\begin{cor} \label{cor1}   Let $F\in {\mathcal F}(n,f_0,T,g,B,\theta)$. Suppose $T\subset[-1,\cos{\theta}]$ and $g(t)\le-a<0$ for all $t\in T$.  Then for every $(N,n,\theta)$ spherical code $C$  we have
$$
A(T)\le \frac{2}{N}\left\lfloor Q\right\rfloor, \quad Q:=\frac{F(1,1,1)+3(N-1)B-f_0N^2}{6a}. \eqno (5)
$$
In particular, if  $T=\{t\}$  and $Q<1$, then $A_t=0.$
\end{cor} 
\begin{proof} It is easy to see that  $\hat h(n,\theta,T,g)\le -A(T)a$. Therefore,  Theorem \ref{thm1} yields  $$A(T)\le 2\frac{Q}{N}.$$ 

Note that 
$
NA(T)=2E(T),
$
where $E(T)$ denote the number of unordered pairs $(u,v)$, $u,v\in C$, with $u\cdot v\in T$. Then  $A(T)=2k/N$, where $k\in{\mathbb Z}$, $k\ge0$. This completes the proof.  
\end{proof}

Arguing as above, we can prove the following corollary. 

\begin{cor} \label{cor2}   Let $F\in {\mathcal F}(n,f_0,T,g,B,\theta)$. Let $a>0$.  Suppose  $T\subset[-1,\cos{\theta}]$ and  $g(t)\le a$ for all $t\in T$.  Then for every $(N,n,\theta)$ spherical code $C$  we have
$$
A(T)\ge \frac{2}{N}\left\lceil {R}\right\rceil, \quad R:=\frac{f_0N^2-F(1,1,1)-3(N-1)B}{6a}.  \eqno (6)
$$
\end{cor} 




 

\medskip

\noindent {\bf Example 1.} Bannai and Sloane \cite{BS} proved the uniqueness of maximum kissing arrangements in dimensions 8 and 24. The main step of the uniqueness theorems is to show that the correspondent distance distribution is unique. Here we use corollaries from this section to prove this fact for dimension 8. We think that this approach can be useful for a proof of the uniqueness of maximum kissing arrangement in four dimensions and  other spherical codes. 

Let 
$$
g_0(t)=(2t-1)t^2(2t+1)^2(t+1), 
$$
$$
g_1(t)=(2t-1+a_1)\,t^2(2t+1)^2(t+1),  \quad g_2(t)=(2t-1)(t^2-a_2^2)(2t+1)^2(t+1),
$$
$$
g_3(t)=(2t-1)\,t^2\left((2t+1)^2-a_3^2\right)(t+1),  \quad g_4(t)=(2t-1)\,t^2(2t+1)^2(t+1-a_4), 
$$
where all $a_i>0$. 

Consider the expansion of $g_i$  in terms  of the Gegenbauer polynomials $G_k^{(8)}$: 
$$
g_i(t)=\sum\limits_{k=0}^6{c_{k,i}G_k^{(8)}(t)}.
$$  
It is known that the coefficients $c_{k,0}>0$ for all $k=0,\ldots, 6$ (see \cite{Lev2, OdS}, \cite[Ch.13]{CS}). We may assume that 
 $$
 c_{k,i}>0 \, \mbox{ for all }\, k \, \mbox{ and } \, i.
 $$  
Indeed, note that if $a_i=0$, then $g_i(t)=g_0(t)$. Therefore, if $a_i$ are small enough then the positivity of $c_{k,0}$  yield the positivity of $c_{k,i}$. 

Let 
$$F_i(x, y, z)=g_i(x)+g_i(y)+g_i(z), \quad B_i=g_i(1), \quad f_{0,i}=3c_{0,i}.$$ 
Since all $ c_{k,i}$ are positive, we have that all $F_i\in \bv(8, f_{0,i})$. 

In dimension 8 the maximum kissing arrangement is a $(240,8,\pi/3)$ spherical code. First we  apply Corollary \ref{cor1} with $F_0$. If $t\in [-1,0.5]$ and $t\ne -1, \pm 0.5, 0$, then $g_0(t)<0$. It is well known  (see \cite{Lev2, OdS}, \cite[Ch.13]{CS}) that 
$$
D:=g_0(1)-240c_{0,0}=0. 
$$ 
Then  $Q=120D/g_0(t)=0$.  Thus,
$$
 \mbox{if } t\in [-1,0.5], \, t\ne t_i,  \mbox{ then } A_t=0, \, \mbox{ where }  t_1:=0.5, \; t_2:=0, \; t_3:=-0.5, \; t_4:=-1.   \eqno (7) 
$$

Let 
$$
T_1:=[0.5-a_1/2,0.5], \; T_2:=[-a_2, a_2], \; T_3:= [-0.5-a_3,-0.5+a_3], \; T_4:=[-1,a_4-1].  
$$
It is clear that $g_i$ achieves its maximum on  $T_i$ at $t_i$ and (7) implies that $A(T_i)=A_{t_i}$.  

It can be proved that $F_i\in {\mathcal F}(8,f_{0,i},T_i,g_i,B_i,\pi/3)$. By Corollary \ref{cor2} we have $A(T_i)\ge P_i$, where $P_i$ denote the right side  of (6). $P_i$  can be found by the direct calculation.  We have $P_1=P_3=56$,  $P_2=126$ and $P_4=1$.  Therefore, $A([-1,0.5])\ge 239$. On the other side, $A([-1,0.5])= 239$. Hence the inequalities $A(T_i)\ge P_i$ are equalities. Thus
$$
A_{-1}=1, \; A_{-1/2}=A_{1/2}=56, \; A_0=126. 
$$



\section{SDP bound for distance graphs}

In this section we consider a version of Theorem \ref{thm1} for distance graphs of spherical codes.

\begin{defn}
 Let  $T \subset[-1,1)$. Let $C$  be a finite subset of\, ${\mathbb S}^{n-1}$. 
Denote by  $\dg(C,T)$ a graph with vertices in $C$ and edges $(c,c')$  with $c\cdot c'\in T$.  Denote by $E(C,T)$ the set of edges of $\dg(C,T)$. 

Let $g$ be a real function on  $T$. Define 
$$
H_g(C,T):=\sum\limits_{(c,c')\in E(C,T)} {g(c\cdot c'}).
$$
 \end{defn}

\begin{defn} Let $g$ be a real function on  $T\subset[-1,1)$. Let $G=(V,E)$ be a simple graph with $N$ vertices. Denote by $Q(G,n,\theta,T)$   the all of spherical codes $C$ such that $C\in Q(n,N,\theta,T)$ and $\dg(C,T)=G$. 
Define 
$$
\tau(G,n,\theta,T,g):=\sup\limits_{C\in Q(G,n,\theta,T)}{H_g(C,T)}
$$
subject to 
$$
\sum\limits_{c'\in C} {g_T(c\cdot c'})\le h_m(n,\theta,T,g), \, m=\deg(c) \, \mbox { in } \, \dg(C,T),  \; \mbox{ for all } \, c\in C. 
$$
 \end{defn}
 
 The following theorem can be proved by the same arguments as Theorem \ref{thm1}. 
 
 \begin{thm}   \label{thm2}   Let $F\in {\mathcal F}(n,f_0,T,g,B,\theta)$.  Let $G$ be a simple graph with $N$ vertices. Then an $(N,n,\theta)$ spherical code $C$ with $\dg(C,T)=G$ satisfies the following inequality
$$
N^2\le \frac{F(1,1,1)+3(N-1)\,B+6\tau(G,n,\theta,T,g)}{f_0}. \eqno (8)
$$
\end{thm}

Actually, Theorem \ref{thm2} gives a stronger  bound  than Theorem \ref{thm1}.  It is clear that 
$$
2\tau\le N\hat h,  \quad \tau:=\tau(G,n,\theta,T,g), \; \hat h:=\hat h(n,\theta,T,g). \eqno (9)
$$
 Note that (8) coincides with (4) only if $2\tau = N\hat h$.  Therefore,  if $2\tau < N\hat h$ then Theorem \ref{thm2} gives a stronger stronger bound  than Theorem \ref{thm1}.  

\medskip 

Let $\theta\le \pi/2$, $m=1,2,\ldots,n$ and 
$$
T_m:=[-1,a_m], \;  a_m\in \left(-\sqrt{(1+(m-1)\cos{\theta})/m},-\sqrt{(1+m\cos{\theta})/(m+1)}\right). 
$$
It is not hard  to prove that in this case  $\mu=A(n,\theta,T_m)=m$ (see \cite{BM},\cite[Theorem 3]{mus08a}). 

Let $C$ be an $(n, N, \theta)$ spherical code. Then $G=\dg(C,T_m)$ is a graph with vertices of degree at most m. Here we consider cases $m-1$ and $m=2$.

If $m=1$, then  $\hat h=h_1$. In this case $G=k_1K_1\cup k_2T_2$, in other words $G$ consists of $k_1$ isolated vertices and $k_2$ connected components with two vertices. Then $k_1+2k_2=N$. We obviously have 
$$
2\tau\le (N-k_1)\,h_1.  
$$  

Let $m=2$. Then  $\hat h=\max\{h_1,h_2\}$ and if $h_1>h_2$, we have $\hat h=h_1$. 

\begin{prop} Let  $G=\dg(C,T_2)$. The number of connected components in $G$  with $i$ vertices we denote by $k_i$, $i=1,2,3.$  If $h_1>h_2$, then 
$$
2\tau \le 2k_{2\,}h_1+(N-k_1-2k_2-k_3)\,h_2.  \eqno (10)
$$
\end{prop}
\begin{proof}
Since the angular length of edges in $G$ is greater than $2\pi/3$, $G$ doesn't contain triangles.   That yields the contribution in $\tau$ of any connected components  with 3 vertices is at most $h_2$.  Thus, 
$$
2\tau \le 2k_{2\,}h_1+2k_{3\,}h_2+(N-k_1-2k_2-3k_3)\,h_2=2k_{2\,}h_1+(N-k_1-2k_2-k_3)\,h_2.
$$
\end{proof}

This proposition shows that for $m=2$ with $h_1>h_2$ inequality (9) becomes an equality  if only and if $N=2k_2$ and $G=k_2K_2$, i.e. $G$ is the disjoint union of $k_2$ edges.




\section{Concluding Remarks}

In conclusion, we outline some applications of Theorems 1 and 2 and their generalizations.

\subsection{Towards a proof of the uniqueness conjecture}
We know that $k(4)=24$ \cite{mus08a}.  However, in dimension 4 the uniqueness of
the maximal kissing arrangement is conjectured to be the 24--cell  but not yet proven. Equivalently, the uniqueness conjecture is the following: 
\smallskip

{\em Let $C$ be a $(24,4,\pi/3)$ spherical code. Then} 
$$
A_{-1}=A_1=1, \; A_{-1/2}=A_{1/2}=8, \; A_0=6, \;  A_t=0, \, t\ne \pm 1, \pm 1/2, 0.  \eqno (11)
$$

Note that in this dimension the equality $A_{-1}=1$ yields (11) \cite{Boy}. 

\smallskip

Denote by $s_d(n)$ the optimal SDP bound on $k(n)$ given by (3) with $\deg(F)=d$ (see \cite{mv10}).   
In the following table it is shown that this minimization problem is a semidefinite program and that every upper bound on $s_d(4)$ provides an upper bound for the kissing number in dimension 4. 

\begin{itemize}

\item $s_7(4) <24.5797$  -- Bachoc \&  Vallentin \cite{bac08a};

\item $s_{11}(4) < 24.10550859$ --   Mittelmann \& Vallentin \cite{mv10};
 
\item $s_{12}(4)< 24.09098111$ \cite{mv10};

\item $s_{13}(4)< 24.07519774$ \cite{mv10};

\item $s_{14}(4) <24.06628391$ \cite{mv10};

\item $s_{15}(4) <24.062758$ -- Machado \&  de Oliveira Filho \cite{mac16};
 
\item $s_{16}(4) <24.056903$  \cite{mac16}.

\end{itemize}

This table show that $s_d$ with $d>12$ is relatively close to 24, $s_d-24<2/N=1/12$.  We think that our approach which is based on Corollaries \ref{cor1} and \ref{cor2}  (see  Example 1) can help to prove (11). Perhaps, using Proposition 1 and its extensions can be proved that $A_{-1}=1$. 




\subsection{Towards a proof of the 24-cell conjecture}

The sphere packing problem asks for the densest packing of ${\mathbb R}^n$ with unit balls. 
In four dimensions, the old conjecture states that a sphere packing is densest when spheres are centered at the points of lattice $D_4$, i.e.  the highest density  $\Delta_4$  is $\pi^2/16$, or equivalently the highest  center density is $\delta_4=\Delta_4/B_4=1/8$.  
For lattice packings, this conjecture was proved by Korkin and Zolatarev in 1872. 
Currently, for general sphere packings the best known upper bound for $\delta_4$ is $0.130587$, a slight improvement on the Cohn--Elkies bound  of $\delta_4<0.13126$, but still nowhere near sharp. 

In \cite{musin2018} we considered the following conjecture:
\smallskip 

\noindent {\bf The 24--cell conjecture.} {\it Consider the Voronoi decomposition of any given packing $P$ of unit spheres in ${\mathbb R}^4$. The minimal volume of any cell in the resulting Voronoi decomposition of $P$ is at least as large as the volume of a regular 24--cell circumscribed to a unit sphere.}

\smallskip
\noindent Note that a proof of the 24-cell conjecture also proves that $D_4$ is the densest sphere packing in 4 dimensions. 

In  \cite[Sect. 4]{mus14} and  \cite[3.3]{musin2018} we considered polynomials $H_k$ that are positive--definite in  ${\mathbb R}^n$. Actually, $H_k$ are polynomials that extend the Bachoc--Vallentin polynomials $S_k$. It is an interesting problem to find generalizations of Theorems \ref{thm1} and \ref{thm2} for sphere packings in  ${\mathbb R}^n$. Perhaps, these bounds for $n=4$ can help to prove the 24--cell conjecture.

\subsection{Extension of the SDP bound for codes in spherical caps}

In \cite{BM} we considered geometric and linear programming bounds on codes in spherical caps.  Bachoc and Vallentin \cite{bac09b} applied the semidefinite programming approach to obtain upper bounds on $A(n,\theta,\phi)$. They compute  upper bounds for the one--sided kissing number $B(n)$ in several dimensions $n$. In particular they proved that $B(8)=183$. It is an interesting problem to extend Theorems 1 and 2 for codes in spherical caps and to prove that $B(5)=32$ and $B(24)=144855$.

\subsection{SDP bound for contact graphs and Tammes' problem}


Let $g_a$ be a monotonically increasing function on  $T_a=[s-a,s]$, $s:=\cos{\theta}$. Suppose $F$ is as in Corollary \ref{cor2}. Then for every $(N,n,\theta)$ spherical code $C$ we have 
$$
|E(C,T_a)|=\frac{1}{2}A(T_a)\ge P_a:=\frac{\left\lceil {R_a}\right\rceil}{N}, \quad R_a:=\frac{f_0N^2-F_a(1,1,1)-3(N-1)B_a}{6g_a(s)}. \eqno (12)
$$


Note that $E(C,T_0)$ is the set of edges of the contact graph of $C$. Then using small $a$ we can find lower bounds for  $|E(C,T_0)|$. Moreover, if $P_a$ approaches the limit $P_0$ as $a$ approaches 0 then $|E(C,T_0)|\ge P_0$. 

\medskip

The following problem was first asked by the Dutch botanist Tammes in 1930:\\
 {\em Find the largest angular separation $\theta$ of a spherical code $C$ in ${\mathbb S}^2$ of cardinality $N$.} \\In other words, \\ {\it How are $N$ congruent, non-overlapping circles distributed on the sphere when the common radius of the circles is as large as possible?}
 
 The Tammes problem is presently solved for only a few values of $N$: for $N=3,4,6,12$ by L. Fejes T\'oth; for $N=5,7,8,9$ by Sch\"utte and van der Waerden;  for $N=10,11$ by Danzer; for $N=24$ by Robinson.; and  for $N=13, 14$ by Musin \& Tarasov \cite{mus12a, MTT14}.

The computer-assisted solution of Tammes' problem for $N=13$ and $N=14$ consists of three parts: 
(i) creating the list $L_N$ of all planar graphs with $N$ vertices that satisfy the conditions of \cite[Proposition 3.1]{MTT14}; 
(ii) using linear approximations and linear programming to remove from the list $L_N$ all graphs that do not satisfy the known geometric properties of the maximal contact graphs \cite[Proposition 3.2]{MTT14}; 
(iii) proving that among the remaining graphs in $L_N$ only one is maximal. 

In fact, the list $L_N$ consists of a huge number of graphs. (For $N=13$ it is about $10^8$ graphs.) We think that the lower bound on the number of edges (12) can essentially reduce the number of graphs in $L_N$.

\subsection{Generalization of the $k$--point SDP bound for spherical codes}

In \cite{mus14} we invented  the $k$--point SDP bound for spherical codes. Note that for $k=2$ that is the classical Delsarte bound.  The 3--point SDP bound was first considered by Bachoc and Vallentin \cite{bac08a}. Recently, this method  with $k=4,5,6$ was apply for upper bounds of the maximum number of equiangular lines in $n$ dimensions \cite{LMOV}. It is an interesting to find generalizations of results in this paper using  the $k$--point SDP bounds and apply these bounds for $s$--distance sets and equiangular lines. 



\medskip

\medskip

\medskip

\medskip

O. R. Musin, School of Mathematical and Statistical Sciences, University of Texas Rio Grande Valley,  One West University Boulevard, Brownsville, TX, 78520.

 {\it E-mail address:} oleg.musin@utrgv.edu

\end{document}